\documentclass[11pt, oneside]{amsart}       
\usepackage{geometry}                        
\usepackage{graphicx}                
                            \usepackage{amsthm}    
\newtheorem{theorem}{Theorem}[section]
\newtheorem{corollary}{Corollary}[section]
\newtheorem{definition}{Definition}[section]
\newtheorem{lemma}[theorem]{Lemma}
\theoremstyle{remark}
\newtheorem*{remark}{Remark}
\usepackage{tcolorbox}

\title{Tamarkin-Tsygan Calculus and Chiral Poisson Cohomology}
\author{Emile Bouaziz}

\begin{document}\begin{abstract} We construct and study some vertex theoretic invariants associated to Poisson varieties, specialising in the conformal weight $0$ case to the familiar package of Poisson homology and cohomology. In order to do this conceptually we sketch a version of the \emph{calculus}, in the sense of \cite{TT}, adapted to the context of vertex algebras. We obtain the standard theorems of Poisson (co)homology in this \emph{chiral} context. This is part of a larger project related to promoting non-commutative geometric structures to chiral versions of such.\end{abstract} 
\maketitle

\section{Introduction} We construct in this note certain rather infinite dimensional invariants associated to Poisson varieties. They can be thought of as a sort of \emph{chiralization} of the Tamarkin-Tsygan calculus associated to a Poisson variety (which calculus computes the calculus associated to a $*$-quantization of a smooth commutative algebra $A$, cf. \cite{Ko}). In particular the invariants produced reduce to cohomology of the chiral de Rham when $\pi$ vanishes and Poisson (co)homology in conformal weight $0$. We also identify the fixed points of the homotopical $S^{1}$ action coming from $d^{ch}_{dR}$ and prove an analogue of a theorem of Brylinski (\cite{Bry}) showing that these invariants are trivial when $\pi$ is non-degenerate. We view this note as a part of a larger set of questions regarding when Tamarkin-Tsygan calculi can be chiralized, the case of matrix factorizations was treated in the paper \cite{Bo}. 

Associated to a smooth $\mathbb{C}$-variety we have the sheaf of Gerstenhaber algebras (henceforth $\mathcal{G}$-algebras, see the introduction of \cite{TT} for a definition), $\Theta_{M}$, of polyvector fields on $M$, endowed with the graded commutative wedge product, $\wedge$, and the Gerstenhaber bracket, $[\,,\,]$. This acts on the sheaf of forms $\Omega_{M}$ via the usual package consisting of contractions, $\iota_{v}$, and Lie derivatives, $\mathcal{L}_{v}$. The various compatibilities between these operations are summarized by saying that $\Theta_{M}$ acts on $\Omega_{M}$ as a $\mathcal{G}$-algebra. Further, $\Omega_{M}$ admits a differential, $d_{dR}$, and one has Cartan's relation $$\big[d_{dR},\,\iota\,\big]=\mathcal{L}.$$ These relations are summarised elegantly in the language of \cite{TT} by saying that the tuple $(\Theta_{M},\Omega_{M},d_{dR})$ forms a \emph{calculus} (a $\mathcal{G}$-algebra together with a module $M$ and a differential $\partial$ relating the contraction and Lie operations as above). 

 If $\pi$ is a Poisson form on $M$, which is to say an element of $\Theta_{M}$ of cohomological degree $2$ satisfying $[\pi,\pi]=0$, then one obtains a cohomological differential $[\pi,-]$ on $\Theta_{M}$ (cf. the work of Lichnerowicz in \cite{Lich}), and a corresponding homological differential $\mathcal{L}_{\pi}$ on $\Omega_{M}$ (cf. the work of Brylinski in \cite{Bry}). The hypercohomology of these complexes will be denoted $H^{*}_{\pi}(M)$ and $H^{\pi}_{*}(M)$, and referred to as \emph{Poisson cohomology} and \emph{Poisson homology} respectively.  Further, one has $\big[d_{dR},\,\mathcal{L}_{\pi}\,\big]=0$, whence one can form the $\mathbb{Z}/2$-graded totalization $(\Omega_{M}((u)),ud_{dR}+\mathcal{L}_{\pi})$ of the \emph{mixed complex} $(\Omega_{M},\mathcal{L}_{\pi},d_{dR})$, where $u$ is a variable of cohomological degree $-2$. The hypercohomology of this totalization will be denoted $H^{\pi}_{*}(M)^{S^{1}}$, (cf. \cite{PTVV}, section 1 for an explanation as to this notational choice). The construction above implies that the tuple $(H^{*}_{\pi}(M),H^{\pi}_{*}(M),d_{dR})$ forms a calculus in the sense of \cite{TT}. Poisson (co)homology is somewhat difficult to compute, it is not even known when it is finite (cf. \cite{ES} for results in this direction), nonetheless one can compute it in the case that $\pi$ is non-degenerate, and in all cases one can compute the totalization $H^{\pi}_{*}(M)^{S^{1}}$. \begin{lemma} (Brylinski, \cite{Bry}). If $\pi$ is non-degenerate then we have $H^{\pi}_{*}(M)\cong H^{d-*}(M)$, where $d$ is the dimension of $M$. Similarly we have $H^{*}_{\pi}(M)\cong H^{*}(M)$. \end{lemma}\begin{proof} We deal with the case of homology. Let $\omega$ be the symplectic form corresponding to $\pi$. Brylinski shows that the associated symplectic Hodge-$*$ operator (which obviously induces an isomorphism between the graded sheaves $\Omega^{*}_{M}$ and $\Omega^{d-*}_{M}$) intertwines the differentials $d_{dR}$ and $\mathcal{L}_{\pi}$, whence the lemma is proven. The case of cohomology follows similarly.\end{proof} \begin{lemma} There is an isomorphism $H^{\pi}_{*}(M)^{S^{1}}\cong H^{*}(M)((u))$. \end{lemma}\begin{proof} This is immediate from the identity $e^{\iota_{\pi}}d_{dR}e^{-\iota_{\pi}}=d_{dR}+\mathcal{L}_{\pi}$, which is a special case of the identity $e^{ad_{x}}y=e^{x}ye^{-x}$, itself a special case of the Baker-Campbell-Hausdorff formula. \end{proof}

The purpose of this note is to promote all of the above structure to a \emph{chiral} such, which is to say to produce a similar package where $\Omega_{M}$ is replaced by the sheaf of vertex algebras, $\Omega^{ch}_{M}$, introduced in \cite{MSV}. Roughly speaking one can think of this as a sort of semi-infinite (with respect to positive loops) version of the above, taking place now on the algebraic loop space, $M((z))$, of the variety $M$. The result will be endowed in particular with a grading by $\mathbb{Z}_{\geq 0}$,  referred to as \emph{conformal weight}, which reproduces the above in conformal weight $0$. 

One could simply write down the complexes by hand, however it is also possible to construct the objects of interest rather cleanly, one should first produce an analogue of the calculus above, in particular this requires working with chiral polyvectors $\Theta^{ch}_{M}$ and constructing on them the structure of what probably deserves to be called a \emph{vertex}-$\mathcal{G}$-algebra. This is a graded vector space $V$ endowed now with \emph{two} $\mathbb{Z}$-indexed families of bilinear operations, $$v\otimes w\mapsto v_{(i)}w,\, v\otimes w\mapsto v_{\{i\}}w,$$ satisfying a number of compatibilities in analogy with the contraction and Lie operations corresponding to the action of a $\mathcal{G}$-algebra. There is a corresponding notion of \emph{vertex calculus}. 

\begin{remark} The construction of the \emph{Gerstenhaber operations,} $$v\otimes w\mapsto v_{\{i\}}w,$$ is in the spirit of the work of Lian and Zuckerberg (\cite{LZ}), the product $\{0\}$  is a special case of the construction of \cite{LZ}. Note however that \cite{LZ} deals with the construction of a $\mathcal{G}$-algebra on the structure of the BRST complex associated to a $\mathcal{N}$=2 vertex algebra, which cohomology vanishes in conformal weights greater than $0$, it is really the higher conformal weight Poisson (co)homology which interests us in this note. \end{remark} The corresponding vertex calculus for the Poisson variety $(M,\pi)$ is then obtained in a manner in exact analogy with the above (non-chiral) procedure, the cohomological differential will be $\pi_{\{0\}}$ and the homological differential the corresponding Lie derivative. 

More explicitly, and with respect to \'{e}tale local coordinates $x^{i}$ on $M$, one has vectors  $x^{i}_{j}, y^{i}_{j+1}, \psi^{i}_{j}, \phi^{i}_{j+1}$, $j\geq 0$, in $\Theta^{ch}_{M}$, where Latin letters denote bosonic vectors and Greek letters fermionic ones. The vectors $\psi$ transform as vector fields, and $\phi$ as forms, the $y$ transform as vector fields plus an extra contribution coming from fermionic vectors (cf. \cite{MSV} for a detailed discussion). If $\pi$ is given with respect to these coordinates as $\pi_{ij}\partial^{i}\partial^{j}$, then the differential on $\Theta^{ch}_{M}$ will be the residue $$\partial^{ch}_{\pi}:=(\partial_{k}\pi_{ij}\psi_{0}^{i}\psi_{0}^{j}\phi_{1}^{k}+\pi_{ij}\psi_{0}^{i}y_{1}^{j}-\pi_{ij}\psi_{0}^{j}y_{1}^{i})_{(0)},$$ which the reader can easily confirm induces the correct differential in conformal weight $0$. Setting $Q:=y^{i}_{1}\phi^{i}_{1}$, we observe that we have $\partial^{ch}_{\pi}=Q_{(0)}\pi$, so that the differential is given by $(Q_{(0)}\pi)_{(0)}$, which the initiated reader will recognise from the construction of \cite{LZ}. The homological differential will be the corresponding Lie derivative, $\mathcal{L}^{ch}_{\pi}$, acting on $\Omega^{ch}_{M}$.

We summarize the results of this note in the following theorem, \begin{theorem}

\begin{itemize}\item \emph{Formal structure}. If $(M,\pi)$ is a smooth Poisson variety then the associated tuple $((\Theta^{ch}_{M},\partial^{ch}_{\pi}),(\Omega^{ch}_{M},\mathcal{L}^{ch}_{\pi}),d^{ch}_{dR})$ defines a sheaf of vertex calculi.\item \emph{De Rham invariants}. The hypercohomology of the $S^{1}$-fixed points are identified with $H^{ch}(M)((u))$, the $2$-periodization of the hypercohomology of the $\Omega^{ch}_{M}$ with vanishing differential. \item \emph{Brylinksi type theorem}. If $\pi$ is non-degenerate then the hypercohomology of $(\Theta^{ch}_{M},\partial^{ch}_{\pi})$ and $(\Omega^{ch},\mathcal{L}^{ch}_{\pi})$ vanish in conformal weight greater than $0$. \end{itemize} \end{theorem}

\section{vertex calculus} \subsection{Chiral Polyvectors} We assume familiarity with the basic theory of vertex algebras, the reader may consult \cite{MSV} for an introduction. The vertex algebras with which we deal will all be \emph{super} such, we adopt the convention where all commutators are assumed to be super-commutators, vertex algebras are to be assumed vertex super-algebras and so on.  We will assume further that the reader is familiar with the construction of the chiral de Rham complex, $\Omega^{ch}_{M}$, as given for example in \cite{MSV}. 

Throughout $x^{i}$ will denote \'{e}tale local coordinates on the variety $M$, as above we have correspoding generating vectors $x^{i}_{0}, y^{i}_{1}, \phi^{i}_{0}, \psi^{i}_{1}$, where lower subscripts denote conformal weight, $x,y$ are bosonic and $\phi, \psi$ are fermionic. Note that the conformal weight $0$ subspace is just the sheaf of forms $\Omega_{M}$. $\Omega^{ch}_{M}$ is a vertex operator algebra for any smooth $M$ and an $\mathcal{N}=2$ such in the Calabi-Yau case. The sheaf of chiral polyvectors, denoted $\Theta^{ch}_{M}$ is defined by conformally re-grading $\Omega^{ch}_{M}$, so that the generating $\phi$-vectors are now of weight $1$, and the $\psi$ vectors of weight $0$. When care is needed we will denote the corresponding vectors $\bar{\psi}$ and $\bar{\phi}$.\begin{remark}This apparently trivial adaptation hides some slight subtleties, the formulae of \cite{MSV} now imply that $\Theta^{ch}_{M}$ is \emph{no longer} endowed with a Virasoro vector compatible with its conformal grading, unless $M$ is Calabi-Yau. Happily, in the CY case $\Theta^{ch}_{M}$ is endowed with a $\mathcal{N}=2$ structure, cf. \cite{Gor}. We note here that, as remarked in \cite{MSV}, all objects exist canonically as \emph{gerbes}.\end{remark}

\begin{definition} On a formal $D$-disc $\Delta^{D}$, we define the vector $\bar{Q}\in\Theta^{ch}_{\Delta^{D}}$ of conformal weight $2$ by $$\bar{Q}=y^{i}_{1}\bar{\phi}^{i}_{1}.$$ \end{definition}
\begin{lemma} The $0$-mode of this vector is invariant under the action of $Aut(\Delta^{D})$, and so $\bar{Q}_{(0)}$ is defined on $\Theta^{ch}_{M}$ for any smooth $M$. \end{lemma}\begin{proof} This follows from the formulae of \cite{MSV}.\end{proof}
\subsection{Brackets}
\begin{definition} We define the Gerstenhaber $i$-products on $\Theta^{ch}_{M}$, $v\otimes w\mapsto v_{\{i\}}w$, by $$v_{\{i\}}w:=(\bar{Q}_{(0)}v)_{(i)}w.$$\end{definition}
Let us note the following basic properties of these products, \begin{itemize} \item The product $v\otimes w\mapsto v_{\{i\}}w$ is of conformal weight $-i$, and so $\{0\}$ preserves the conformal weight $0$ subspace. \item The product is of fermion number $-1$, so that if $v$ and $w$ are of cohomological degree $n$ and $m$ respectively, then $v_{\{i\}}w$ is of cohomological degree $n+m-1$. \item \'{E}tale locally on $M$, the product $v\otimes w\mapsto v_{\{i\}}w$ precisely measures the failure of the conformal weight $0$ mode, $\bar{Q}_{(1)}$, to be a derivation with respect to the $(i-1)$-product. Indeed, this follows from a Borcherds-Jacobi identity, $$[\bar{Q}_{(1)},v_{(j-1)}]-(\bar{Q}_{(1)}v)_{(j-1)}=(\bar{Q}_{(0)}v)_{(j)},$$ so that $\bar{Q}_{(1)}$ provides a sort of BV-operator for the product $_{\{j\}}$. In particular this BV-operator exists when $M$ is Calabi-Yau\end{itemize}

\begin{lemma} The product $v\otimes w\mapsto v_{\{0\}}w$ restricts on the conformal weight $0$ subspace to the usual Gerstenhaber bracket of polyvector fields.\end{lemma}
\begin{proof} $v_{\{0\}}=(\bar{Q}_{(0)}v)_{(0)}$ acts a derivation with respect to all $(j)$-products, in particular with respect to the $(-1)$-product. We must then verify $(\bar{\psi}^{i}_{0})_{\{{-1}\}}x^{j}_{0}=(-x^{j}_{0})_{\{-1\}}\bar{\psi}^{i}_{0}=\delta_{ij}.$  This follows immediately from $\bar{Q}_{(0)}(\bar{\psi}^{i}_{0})=y^{i}_{1}$ and $\bar{Q}_{(0)}x^{i}_{0}=\bar{\phi}^{i}_{1}$.\end{proof}

\begin{corollary} If $\pi$ is a Poisson form then the operator $\pi_{\{0\}}$ defines on $\Theta^{ch}_{M}$ a cohomological differential, which is moreover a derivation with respect to all $(j)$-products, restricting to the Poisson cohomology differential on the subspace of conformal weight $0$.  \end{corollary}
\begin{proof} We compute $$2\pi_{\{0\}}^{2}=\big[\pi_{\{0\}},\pi_{\{0\}}\big]:=\big[(\bar{Q}_{(0)}\pi)_{(0)},(\bar{Q}_{(0)}\pi)_{(0)}\big],$$ $$\big[(\bar{Q}_{(0)}\pi)_{(0)},(\bar{Q}_{(0)}\pi)_{(0)}\big]=((\bar{Q}_{(0)}\pi)_{(0)}\bar{Q}_{(0)}\pi)_{(0)}.$$ Now we note $$\big[\bar{Q}_{(0)},\,(\bar{Q}_{(0)}\pi)_{(0)}\,\big]= (\bar{Q}_{0}^{2}\pi)_{0}=0,$$  whence we deduce that $$((\bar{Q}_{(0)}\pi)_{(0)}\bar{Q}_{(0)}\pi)_{(0)}=(\bar{Q}_{(0)}(\bar{Q}_{(0)}\pi)_{(0)})_{(0)}:=(\pi_{\{0\}}\pi)_{\{0\}}=0.$$ We have seen above that $\{0\}$ restricts to the Gerstenhaber bracket on polyvector fields, whence the induced map on the conformal weight $0$ subspace is as claimed.  \end{proof}

\begin{remark} The identity $\big[\pi_{\{0\}},\,\pi_{\{0\}}\,\big]=(\pi_{\{0\}}\pi)_{\{0\}}$ is the special case of a general identity which is valid in any vertex $\mathcal{G}$-algebra, $$\big[v_{\{i\}},\,w_{\{j\}}\,\big]=\sum_{k=0}^{\infty}\binom{i}{k}\big(v_{\{k\}}w\big)_{\{i+j-k\}}$$ which shows that in the language of \cite{He} the Gerstenhaber products satsify the axioms of a (shifted) $Lie^{*}$-algebra.\end{remark}

\begin{remark} The above differential graded vertex algebra contains a large commutative sub-algebra, generated by the $x$ and $\psi$ variables. This is simply the space of arcs into the derived variety $T^{*}_{\pi}[-1]M$. According to a lemma of Arakawa (\cite{A}), such has the structure of a (shifted) \emph{Poisson vertex algebra}, the Gerstenhaber products $v\otimes w\mapsto v_{\{i\}}w$, for $i\geq 0$, recover this structure. \end{remark}

\begin{definition} A \emph{vertex} $\mathcal{G}$-\emph{algebra} is defined to be a vertex algebra $V$ endowed with odd bilinear operations of $v\otimes w\rightarrow v_{\{i\}}w$, of conformal weight $-i$, satisfying the following quadratic relations (with the vertex products). \begin{itemize}\item \emph{Lie* structure}, $$\big[v_{\{i\}},w_{\{j\}}\big]=\sum_{k=0}^{\infty}\binom{i}{k}\big(v_{\{k\}}w\big)_{\{i+j-k\}},$$ \item \emph{Generalized derivation property,} $$\big[v_{\{i\}},w_{(j)}\big]=\sum_{k=0}^{\infty}\binom{i}{k}\big(v_{\{k\}}w\big)_{(i+j-k)},$$  \item \emph{Bracket of a product,} $$(v_{(i)}w)_{\{j\}}=\sum_{k=0}^{\infty}(-1)^{k}\binom{i}{k}\big(v_{\{i-k\}}v_{j+k}+u_{(i-k)}v_{\{j+k\}}-(-1)^{k}(v_{\{i+j-k\}}u_{(k)}+v_{(i+j-k)}u_{\{k\}})\big).$$\end{itemize}\end{definition}

\begin{remark} We have included the above definition for completeness, and because it may be of interest to some readers, we do not strictly require the full structure above for the results presented in this note. \end{remark}

The above discussion can now be summarized in the following lemma, \begin{lemma} The products $v\otimes w\mapsto v_{\{i\}}w:=(\bar{Q}_{(0)}v)_{(i)}w$  endow the sheaf of chiral polyvector fields, $\Theta^{ch}_{M}$, with the structure of a vertex $\mathcal{G}$-algebra.\end{lemma}
\begin{proof} All of the above identities follow from Borcherds-Jacobi identities either by substition or by commuting with $\bar{Q}_{(0)}$. The only properties of $\bar{Q}_{0}$ which we required are that it be of approrpriate conformal weight (2) and parity (odd), and that $\bar{Q}_{(0)}^{2}=0$. \end{proof}

\subsection{Action on Chiral Forms} We construct now an action of the sheaf of $\mathcal{G}$-vertex algebras $\Theta^{ch}_{M}$ on $\Omega^{ch}_{M}$. This amounts to the data of a family of contraction and Lie derivative operators, each $\mathbb{Z}$-indexed. That is to say, $v\in\Theta^{ch}_{M}$ determines for each $j\in\mathbb{Z}$, operators $\iota^{v}_{j}$ and $\mathcal{L}^{v}_{j}$, satisfying a bunch of identities and compatibilites with the bracket products $v\otimes w\mapsto v_{\{i\}}w$. We will not go into laborious detail regarding this as it is rather clear how to define a representation of a vertex $\mathcal{G}$-algebra, one demands for example that the following holds, $$\big[\mathcal{L}^{v}_{i},\,\mathcal{L}^{w}_{j}\,\big]=\sum_{k=0}^{\infty}\binom{i}{k}\mathcal{L}^{v_{\{k\}}w}_{i+j-k}.$$ 

We will of course also wish to have an analogue of the Cartan formula, in general this is a special property of a given representation of a (vertex), $\mathcal{G}$-algebra. Recall that when such a formula is available, one says that the resulting data forms a calculus. 
The action by Lie derivatives constructed will in particular specialize to the infinitesimal form of the action of automorphisms produced in \cite{MSV}. If $V$ is a vertex super-algebra we will write $Lie(V)$ for the Lie super-algebra generated by the modes of fields of $V$. We begin with the following trivial lemma, \begin{lemma} There is an isomorphism of sheaves of Lie algebras $Lie(\Theta^{ch}_{M})\rightarrow Lie(\Omega^{ch}_{M})$. \end{lemma} \begin{proof}Dispensing of superscripts to unburden notation, and working formally locally, we send the modes $x_{(i)}, y_{(i)}\in Lie(\Theta^{ch}_{M})$ to the identically denoted elements of $Lie(\Omega^{ch}_{M})$. We stipulate further $\bar{\phi}_{(j)}\mapsto \phi_{(j-1)}$ and $\bar{\psi}_{(i)}\mapsto \psi_{(i+1)}$. This can easily be checked to extend to an isomorphism of Lie algebras. \end{proof}

\begin{corollary} $\Omega^{ch}_{M}$ is a module for the sheaf of vertex algebras $\Theta^{ch}_{M}$. \end{corollary}

\begin{definition} For $v\in\Theta^{ch}_{M}$ a local section, the contraction operators $\iota^{v}_{j}$ are defined to be the $j$-modes of the action of $\Theta^{ch}_{M}$ on $\Omega^{ch}_{M}$ constructed above.  The Lie derivative $\mathcal{L}^{v}_{j}$ is \emph{defined} to be $\big[d^{ch}_{dR},\iota^{v}_{j-1}\big]$.\end{definition}

Now recall that in the data of a calculus there is a homotopy relating contractions to Lie derivatives, one can adapt this definition to the vertex context in an evident manner.  \begin{definition} A tuple consisting of a vertex $\mathcal{G}$-algebra $V$, a module $M$ for it, and an odd square zero derivation $\partial$ satisfying the following \emph{Cartan formula}, for all $v\in V,j\in\mathbb{Z}$, $$\mathcal{L}^{v}_{j}=\big[\partial,\,\iota^{v}_{j-1}\,\big],$$ is said to be a \emph{vertex calculus}. \end{definition}

\begin{remark} If $\eta$ is a vector field on $M$, given in local coordinates as $\eta_{i}\partial^{i}$, one finds that the $0$ Lie derivative action corresponding to $\eta$ is precisely the residue $$(\eta_{i}y^{i}_{1}+\partial_{j}\eta_{i}\phi^{j}_{0}\psi^{i}_{1})_{(0)},$$ wherein one sees (the infinitesimal form of) the famous observation that fermions \emph{cancel the anomaly}, cf. \cite{MSV}. \end{remark}

The above discussion is summarized as follows, the reader may either assume $M$ is Calabi-Yau or else reformulate the following as a statement about gerbes; \begin{lemma} The chiral de Rham differential $d^{ch}_{dR}$ endows the module $\Theta^{ch}_{M}$-module $\Omega^{ch}_{M}$ with the structure of a sheaf of vertex calculi . \end{lemma} \begin{proof} Once it is shown that the operators satisfy the axioms of a representation of a vertex $\mathcal{G}$-algebra, it will follow by construction that $d^{ch}_{dR}$ endows this with the structure of a calculus. That the $\iota$ and $\mathcal{L}$ operators satisfy the requisite axioms can be checked easily from the definitions. \end{proof}

\begin{corollary} The operator $\big[d^{ch}_{dR},\,\iota^{\pi}_{0}\,\big]=\mathcal{L}^{\pi}_{1}$, denoted henceforth $\mathcal{L}^{ch}_{\pi}$ and referred to as the \emph{chiral Poisson differential}, is of square zero and cohomological degree $-1$. The conformal weight $0$ subspace reproduces the usual Poisson homology complex of \cite{Bry}. There is an additional differential, given by $d^{ch}_{dR}$, which commutes with $\mathcal{L}^{ch}_{\pi}$.\end{corollary}

The following definition-lemma summarizes the above discussion; \begin{definition} The hypercohomology $$H_{\pi}^{ch}(M):=H^{*}(M,(\Theta^{ch}_{M},\partial^{ch}_{\pi})),$$ will be referred to as \emph{chiral Poisson cohomology}. The hypercohomology $$H^{\pi}_{ch}(M):=H^{*}(M,(\Omega^{ch}_{M},\mathcal{L}^{ch}_{M})),$$ will be referred to as \emph{chiral Poisson homology}. \end{definition}

\begin{lemma}The formulae defined above endow the triple $(H^{ch}_{\pi}(M),H^{\pi}_{ch}(M),d^{ch}_{dR})$ with the structure of a vertex calculus.\end{lemma} \begin{proof} Above we have constructed everything on the sheaf level, it formally descends to the level of hypercohomology. \end{proof}\section{ Chiral Poisson (co)homology} \subsection{General Results} We now turn to the task of proving the expected basic theorems concerning chiral Poisson (co)homology. As was mentioned in the introduction Poisson (co)homology is a somewhat subtle invariant of a Poisson variety $(M,\pi)$ and as such we of course cannot expect to compute its chiral analogue too easily, as this chiral analogue is \emph{at least} as intractable as the classical version. 

A benefit to the somewhat lengthy discussion above is that the basic expected properties of chiral Poisson (co)homology can now be verified quite easily. We begin with the identification of the fixed points of the $S^{1}$-action.  \begin{lemma} There is an isomorphism $$H^{\pi,S^{1}}_{ch}(M):=H^{*}(M,\big(\Omega^{ch}_{M}((u)),ud^{ch}_{dR}+\mathcal{L}^{ch}_{\pi}\big))\cong H^{*}_{dR}(M)((u)),$$ in particular there are no non-zero classes of strictly positive conformal weight. \end{lemma} 
\begin{proof} Recall that $\mathcal{L}^{ch}_{\pi}$ satisfies $\big[d^{ch}_{dR},\iota^{\pi}_{0}\big]$. Now note that $\iota^{\pi}_{0}$ is of even cohomological degree $-2$ and further is locally nilpotent, as the cohomological degree is bounded below on each fixed conformal weight piece of $\Omega^{ch}_{M}$. It follows that $$\exp\big(\frac{\iota^{\pi}_{0}}{u}\big)$$ is a well-defined operator on this $2$-periodic complex. This conjugates the differential, $$d^{ch}_{dR}+u^{-1}\mathcal{L}^{ch}_{\pi},$$ to the usual chiral de Rham differential. One then applies the results of \cite{MSV} to note that the resulting hypercohomology is simply $H_{dR}(M)((u))$ in conformal weight $0$. \end{proof}

We now prove the analogue of the theorem of Brylinksi (\cite{Bry}) showing that these invariants are really invariants of the singularities of the form $\pi$, which is to say that one obtains nothing of interest when $\pi$ is non-degenerate.\begin{theorem} If $\pi$ is non-degenerate, then there is an isomorphism $H^{\pi}_{ch}(M)\cong H^{d-*}_{dR}(M)$, placed in conformal weight $0$, where $d=2n$ is the dimension of $M$. A similar result holds for Poisson cohomology. In particular there are no classes of non-zero conformal weight. \end{theorem}\begin{proof} Let $\omega$ be the symplectic form dual to $\pi$. If $x$ is a point of $M$ then there are formal coordinates around $x$ where $\pi$ is in standard Darboux form. (We caution the reader that this is not true in the \'{e}tale topology, indeed it fails for the form $d\log(x^{1})d\log(x^{2})$ on $\mathbb{C}^{*}\times\mathbb{C}^{*}$.) 

Now, the machinery of Gelfand-Kazhdan formal geometry (cf. \cite{MSV} for an introduction) implies that there is an associated torsor $\widehat{M}\rightarrow M$ for the (pro-) group scheme $Symp(\Delta^{2n},\omega_{std})$ of formal symplectomorphisms of the $2n$-disc with its standard symplectic form. The group $Symp(\Delta^{2n},\omega_{std})$ acts on $(\Omega^{ch}_{\Delta^{2n}},\mathcal{L}^{ch}_{\pi^{-1}})$ and $(\Omega^{ch}_{M},\mathcal{L}^{ch}_{\pi})$ is obtained by reduction along this torsor. Recalling that Brylinksi's result computes the conformal weight $0$ subsapce, we are thus reduced to showing that the inclusion of the weight $0$ subspace $$(\Omega_{\Delta^{2n}},\mathcal{L}_{\omega^{-1}})\longrightarrow (\Omega^{ch}_{\Delta^{2n}},\mathcal{L}^{ch}_{\omega^{-1}}),$$ is a quasi-isomorphism.

It suffices to handle the case of $n=1$.  Now let $x^{1},x^{2}$ be coordinates on $\Delta^{2}$. Consider the vector $$H:=x^{1}_{1}\phi^{2}_{0}-x^{2}_{1}\phi^{1}_{0}\in\Omega^{ch}_{\Delta^{2}}.$$ This has conformal weight $1$ and cohomological degree $1$. Observe now the following simple identity $$H_{(0)}(y^{1}_{1}\psi^{2}_{1}-y^{2}_{1}\psi^{1}_{1})=x^{i}_{1}y^{i}_{1} +\phi^{i}_{1}\psi^{i}_{1},$$ where of course we recognise $x^{i}_{1}y^{i}_{1}+\phi^{i}_{1}\psi^{1}_{1}$ as the Virasoro vector $L$. Let us note further that we have, by construction of $\mathcal{L}^{ch}_{\omega^{-1}}$, that $$(y^{1}_{1}\psi^{2}_{1}-y^{2}_{1}\psi^{1}_{1})_{(1)}=\mathcal{L}^{ch}_{\omega^{-1}}.$$ Now $H_{(0)}$ acts as a derivation with respect to all $(j)$-products whence we compute $$\big[H_{(0)},\,\mathcal{L}^{ch}_{\omega^{-1}}\,\big]=L_{(1)},$$ so that $L_{(1)}$ acts trivially on cohomology. Now this operator is simply the (diagonalisable) grading operator for the conformal grading, whence the theorem is established. 
 \end{proof}

\begin{remark} When $M$ is Calabi-Yau such that the associated volume form $vol_{M}$ is \emph{compatible} with $\pi$ in the sense that $\mathcal{L}_{\pi}(vol_{M})=0$, then the $\mathcal{N}=2$ algebra acts on $\Omega^{ch}_{M}$ compatibly with the differential $\mathcal{L}^{ch}_{\pi}$. In particular this is the case for symplectic varieties, in which case there is in fact an $\mathcal{N}=4$ action, the operator $H$ above is induced from this structure. \end{remark}

\subsection{An extended example} We will treat in some detail the example of $M=\mathbb{C}^{2}$ equipped with the form $\pi:=x^{2}\partial_{1}\partial_{2}$. This is an analytic local model for \emph{local} Poisson surfaces constructed from the data of a curve $\Sigma$ with a non-vanishing vector field $\xi$, together with a line bundle, $L$, such a datum specifies a $\mathbb{C}^{*}$-equivariant Poisson structure on the total space of $L$. In the example of $M=\mathbb{C}^{2}$ the $\mathbb{C}^{*}$ action corresponds to giving $x^{2}$ weight $1$. Let us first compute the Poisson homology of $(\mathbb{C}^{2},\pi)$. We see immediately that $H^{\pi}_{2}=0$, $H^{\pi}_{0}\cong \mathbb{C}[x^{1}]$. There is an evident map $\mathbb{C}[x^{1}]dx^{1}\rightarrow H^{\pi}_{1}$ which one can prove directly is an isomorphism.

A slicker way to do this is as follows, one notes that the Euler vector field, $\eta$, for the $\mathbb{C}^{*}$ action constructed above, is in the image of the map $\pi :\Omega^{1}_{M}\rightarrow T_{M}$, indeed $\eta=x^{2}\partial^{2}=\pi(dx^{1})$. A simple computation now confirms that we have $$\big[\mathcal{L}_{\pi},\,\wedge dx^{1}\,\big]=\mathcal{L}_{\eta}.$$ Now $\mathcal{L}_{\eta}$ is simply the grading operator on homology, we deduce that all the homology comes from the weight $0$ subspace, and thus the above is proven. \begin{remark} We caution the reader that it is not a general fact that a connected algebraic group acting on the Poisson variety $(M,\pi)$ must act trivially on $H^{\pi}(M)$, indeed one can take $\pi=0$ in which case one is dealing with Hodge cohomology, on which the group can certainly act non-trivially for (say) an affine $M$. By the above argument this is true if the infinitesimal action is given by Hamiltonians of functions on $M$. This is in fact a formal consequence of the calculus structure, we are acting on homology classes by vanishing cohomology classes, see \cite{Bry} (3.4) for some related discussions. \end{remark}

Now one might hope that this argument could be applied to easily compute the chiral Poisson homology of $(\mathbb{C}^{2}, x^{2}\partial_{1}\partial_{2})$. The calculus developed in the previous section readily implies that we have $$\big[\mathcal{L}^{ch}_{\pi},\,(dx^{1})_{(-1)}\,\big]=\mathcal{L}^{\eta}_{0},$$ and of course $\mathcal{L}^{\eta}_{0}$ is still the grading operator on homology. However the weight $0$ subspace is now huge, as the annihilation vectors $y^{2}$ and $\psi^{2}$ are now of negative weight.

This at least cuts down the size of the space we must compute with somewhat, and for example we can now compute the conformal weight $1$ component.
Let us enumerate the $\mathbb{C}^{*}$ weight $0$ variables, they are generated over $\mathbb{C}[x^{1}_{0},\phi^{1}_{0}\,]$ by the vectors \begin{itemize}\item $x^{2}_{0}\,\psi^{2}_{1}$  in cohomological degree $-1$, \item $x^{2}_{0}y^{2}_{1}, \,\phi^{2}_{0}\psi^{2}_{0}, \,x^{1}_{1}, \,y^{1}_{1}$ in cohomological degree $0$. \item $y_{1}^{2}\phi_{0}^{2}, \,\phi^{1}_{1}$ in cohomological degree $1$, \item and vanishing in all other cohomological degrees. \end{itemize}

We can compute by hand the following differentials, \begin{itemize} \item $y^{2}_{1}\phi^{2}_{0}\mapsto y^{1}_{1}, \, \phi^{1}_{1}\mapsto x^{2}_{0}y^{2}_{1}+\phi^{2}_{0}\psi^{2}_{1}$, \item $x^{2}_{0}y_{1}^{2}\mapsto -\psi^{1}_{1},\, \phi^{2}_{0}\psi^{2}_{1}\mapsto \psi^{1}_{1}, \, x^{1}_{1}\mapsto x^{2}_{0}\psi^{2}_{1}, \, y^{1}_{1}\mapsto 0$, \item and vectors of cohomological degree $-1$ map to $0$ for trivial reasons. \end{itemize}

Staring at the above one deduces that there is no cohomology in conformal weight $1$ and it is perhaps thus tempting to conjecture that the same is true in all non-zero conformal weights. This is in fact \emph{false}, there are non-zero classes already in conformal weight $2$. One such is given by the vector $$v:= x^{1}_{1}\psi^{1}_{1}-x^{2}_{1}\psi^{2}_{1},$$ as the reader can confirm. Taking products of the above example one thus sees that there can be an arbitrarily long string of vanishing cohomology groups in conformal weights $1,2,....,N$ before some non-zero classes show up.

\begin{remark} When $M$ is endowed with a $\pi$-compatible volume form then there is a Virasoro element $L$ in (global sections of) $\Theta^{ch}_{M}$ which represents a class in $H^{ch}_{\pi}(M)$, the chiral Poisson cohomology. Now the $\{0\}$ product of this class acting on cohomology (resp. the $0$-Lie derivative acting on homology) give the gradings on cohomology and homology respectively, whence we see that this class is precisely the obstruction to non-zero classes of non-zero conformal weight in the case of Calabi-Yau Poisson varieties. \end{remark}

\end{document}